\documentclass[12pt]{amsart}
\usepackage[margin=1in]{geometry} 
\usepackage{SSdefn}
\usepackage{enumitem}
\usepackage{fdsymbol}
\usepackage{mathtools}
\title[Non-noetherian polynomial functors]{Ideal-theoretic non-noetherianity of polynomial functors in positive characteristic}
\author{Karthik Ganapathy}
\address{Department of Mathematics, University of California, San Diego, CA}
\email{\href{mailto:kganapathy@ucsd.edu}{kganapathy@ucsd.edu}}
\thanks{}
\urladdr{\url{https://sites.google.com/view/karthik-ganapathy}}

\theoremstyle{plain}
\newtheorem{mainthm}{Theorem}

\begin{document}
\begin{abstract}
A long-standing open problem in representation stability is whether every finitely generated commutative algebra in the category of strict polynomial functors satisfies the noetherian property. In this paper, we resolve this problem negatively over fields of positive characteristic using ideas from invariant theory. Specifically, we consider the algebra $P$ of polarizations of elementary symmetric polynomials inside the ring of all multisymmetric polynomials in $p \times \infty$ variables. We show $P$ is not noetherian based on two key facts: (1) the $p$-th power of every multisymmetric polynomial is in $P$ (our main technical result) and (2) the ring of multisymmetric polynomials is Frobenius split.
\end{abstract}
\maketitle
\section{Introduction}
Given a strict polynomial functor $X$ over a field $k$, the inverse limit $X_{\infty}$ of the affine schemes $X(k^n)$ is a $\GL_{\infty}$-equivariant affine scheme that exhibits surprising finiteness properties. For instance, when $X$ is a finite direct sum of (Frobenius twists of) the identity functor, the scheme $X$ is $\GL$-noetherian, i.e., dcc holds for $\GL_{\infty}$-stable closed subschemes. More generally, for all finite length $X$, Draisma \cite{dra19top} proved that dcc holds for $\GL_{\infty}$-stable closed \textit{subsets} of the underlying topological space $|X_{\infty}|$. Several authors \cite{bdes23pol, bds25charp} have since developed this theory and subsequently used it to study asymptotic properties of tensors and related objects \cite{bdes23tensors, hs24tensor}. An important open problem is whether the scheme $X_{\infty}$ is $\GL$-noetherian for all finite length $X$. In this paper, we \textit{negatively} answer this question over fields of positive characteristic.
For the remainder of the paper, we work with $\GL$-algebras which are the coordinate rings of the spaces introduced above; $k$ will always be a field of characteristic $p > 0$.

The finite group $\fS_p$ acts on the $\GL$-algebra $R = \Sym(k^p \otimes k^{\infty})$. The invariant subring $R^{\fS_p}$ is the ring of multisymmetric polynomials in $p \times \infty$ variables. We let $P$ be the $\GL$-algebra of \textbf{polarizations}, i.e., the subalgebra of $R$ generated by the $\GL_{\infty}$-orbits of the elementary symmetric polynomials.

\begin{mainthm}\label{thm:glalg}
The $\GL$-algebra $P$ is not $\GL$-noetherian.
\end{mainthm}
\begin{remark}
   We \cite{gan24non} had previously shown that $\Sym(\lw^2(k^{\infty}))$ is not $\GL$-noetherian when $\chark(k) = 2$. However, the ad hoc construction in loc.~cit.~is very specific to the characteristic. The idea can ultimately be traced to Vaughan-Lee's construction \cite{vl75vars} of a variety of abelian-by-nilpotent Lie algebras that is not \textit{finitely based} in characteristic two; such varieties are finitely based in every other characteristic. However, see Remark~\ref{ss:rmks}(1).
\end{remark}
Our proof of Theorem~\ref{thm:glalg} hinges on the next result. To place it in context, we recall the notion of separating invariants introduced by Derksen--Kemper \cite{dk02cit}. For a finite group $G$ acting on a finite-dimensional $k$-vector space $V$, a subalgebra $S \subset k[V]^G$ is \textbf{separating} if the elements of $S$ distinguish $G$-orbits. Draisma--Kemper--Wehlau \cite{dkw08sep} proved that the subalgebra $P_n \subset k[V^{\oplus n}]^G$ generated by the polarizations (i.e., $\GL_n$-orbit) of a generating set of $k[V]^G$ is a separating subalgebra for the diagonal action of $G$ on $V^{\oplus n}$. Grosshans \cite{gross07vect} subsequently showed that $k[V^{\oplus n}]^G$ is integral over $P_n$, and using a lemma of van der Kallen \cite{vdK05err}, observed that $k[V^{\oplus n}]^G$ is precisely the $p$-root closure of $P_n$ in $k[V^{\oplus n}]$, i.e., there exists $d(n) \in \bN$ such that for all $f \in k[V^{\oplus n}]^G$ the element $f^{p^{d(n)}} \in P_n$. We prove $d(n) = 1$ for the permutation representation of $\fS_p$ in characteristic $p$.
\begin{mainthm}\label{thm:pthroot}
Given $f \in R^{\fS_p}$, the element $f^p \in P$.
\end{mainthm}
Our proof does not logically rely on the works cited in the preceding paragraph. Theorem~\ref{thm:glalg} now easily follows because $R$ is graded $F$-split (Lemma~\ref{lem:Fsplmon}) but not finitely generated (Theorem~\ref{thm:rydh}).

\begin{remark}
Although we don't pursue it here, the geometry of $\Spec(P)$ is interesting in its own right, being the inverse limit of $Y_n \coloneqq \Chow_{0, p}(\bA^n \hookrightarrow \bP^n)$, the Chow variety parameterizing zero-dimensional cycles of degree $p$ for the natural embedding $\bA_k^{n} \to \bP_k^{n}$ (see \cite{rydh08phd} for a comprehensive introduction).
For $n \gg 0$, Neeman \cite{nee91cyc} proved $Y_n$ is not normal in char $p$, essentially because $P \ne R^{\fS_p}$.
\end{remark}

Despite the vast literature on multisymmetric polynomials from different viewpoints, a quantitative/uniformity result in the spirit of Theorem~\ref{thm:pthroot} appears to be new. 
We were motivated to investigate this upon 
realizing that an affirmative answer to the noetherian problem in positive characteristic would imply that $d(n)$, as defined above, cannot be bounded (see Remark~\ref{ss:rmks}(2)). 
This is quite ironic given that in characteristic zero, Weyl's theorem on polarizations \cite{weyl39} can be obtained as a consequence of noetherianity for certain $\GL$-algebras (see for example, \cite[Section~2]{sno13delta}). 
We conclude by raising a question suggested by Theorem~\ref{thm:pthroot} which would be a very compelling replacement for Weyl's theorem in the modular setting.
\begin{question}
Given a finite-dimensional representation $V$ of a finite group $G$, does there exist $d \in \bN$ such that $f^{p^d}$ lies in the subalgebra generated by polarizations of $k[V]^G$ for all $f \in k[V^{\oplus n}]^G$ and $n \in \bN$? 
\end{question}

\subsection*{Conventions} For $n \in \bN_{>0}$, we let $[n] = \{1, 2, \ldots, n\}$. All tensor products are over $k$ which is a field of characteristic $p > 0$. For an $\bN$-graded algebra $A$, the ideal generated by homogeneous elements of degree $>n$ is $A_{>n}$. The category of strict polynomial functors is denoted $\Pol_k$.

\subsection*{Acknowledgements} I am grateful to Suchitra Pande and Sridhar Venkatesh for helpful discussions on char $p$ methods and invariant theory, and to Steven Sam and Andrew Snowden for their encouragement.

\section{Multisymmetric polynomials}
\subsection{Shuffle product and divided powers}
The discussion in this section is modeled on \cite{ryd07gens} although our exposition is purposefully far less general. 
Let $K$ be a commutative ring and $S$ be a commutative $K$-algebra which is flat over $K$.
The graded $K$-module of \textbf{symmetric tensors} or \textbf{divided powers}
\[\Gamma(S) = \bigoplus_d \Gamma^d(S) = \bigoplus_d (S^{\otimes d})^{\fS_d}\]
can be endowed with the \textbf{shuffle product} $\times$:
given two symmetric tensors $x \in \Gamma^d(S)$ and $y \in \Gamma^e(S)$, we define
\[ x \times y = \sum_{\sigma \in \fS_{d, e}} \sigma (x \otimes y), \]
where $\fS_{d, e}$ is the set of all shuffle permutations, i.e., all $\sigma \in \fS_{d+e}$ such that $\sigma(1) < \sigma(2) < \ldots < \sigma(d)$ and $\sigma(d+1) < \sigma(d+2) <  \ldots < \sigma(d+e)$.
This endows $\Gamma(S)$ with the structure of a commutative $\bN$-graded $K$-algebra.

Fix $d \in \bN$. Each $\Gamma^d(S)$ has an internal multiplication induced from the usual $K$-algebra structure on $S^{\otimes d}$. We have the assignment $\gamma^d \colon S \to \Gamma^d(S)$ where $\gamma^d(s) = s^{\otimes d}$, which satisfies
\begin{itemize}
    \item $\gamma^d(\lambda s) = \lambda^d \gamma^d(s)$;
    \item $\gamma^d(s+t) = \sum_{d_1 + d_2 = d} \gamma^{d_1}(s) \times \gamma^{d_2}(t)$; and
    \item $\gamma^d(s) \times \gamma^e(s) = \binom{d+e}{e} \gamma^{d+e}(s)$;
\end{itemize}
for all $d, e \in \bN$, $\lambda \in K$ and $s, t \in S$.

\subsection{Multisymmetric polynomials}
We return to the case where $k$ is a field of $\chark(k) = p > 0$ as in the introduction.  Let $\bV$ be the $k$-vector space with basis $\{e_i \}_{i \in \bN_{>0}}$ and $\GL$ be the group of $k$-linear automorphisms of $\bV$.
Set $S = \Sym(\bV)$. We identify $S$ with the polynomial algebra $k[x_1, x_2, \ldots]$ and $R = S^{\otimes p}$ from the introduction with 
the polynomial algebra with variables being the entries of a generic $p\times \infty$ matrix, i.e., 
$ R = k
\left[ \begin{smallmatrix}
   x_{1,1} & x_{1,2} & \cdots & x_{1, n} & \cdots \\
   x_{2,1} & x_{2,2} & \cdots & x_{2, n} & \cdots \\
   \vdots & & \ddots & & \vdots \\
   x_{p,1} & x_{p,2} & \cdots & x_{p, n} & \cdots
\end{smallmatrix} \right]
$.

The ring of \textbf{multisymmetric polynomials} in $p \times \infty$ variables is the $k$-algebra $\Gamma^p(S)$ equipped with the internal product. 
Under the identification of $S^{\otimes p}$ with the polynomial ring above, the action of $\fS_p$ permutes the variables in each column, and $\Gamma^p(S)$ is the invariant subring for this action by definition.

\begin{example}
We shall use $p=3$ as a running example. The element $\gamma^1(x_1^3 x_2^2) \times \gamma^{2}(1)$ is $x_{1,1}^3 x_{1,2}^2 + x_{2,1}^3x_{2,2}^2 + x_{3,1}^3x_{3,2}^2$ and $\gamma^2(x_4) \times \gamma^1(1)$ is $x_{1,4}x_{2,4} + x_{1,4}x_{3,4} + x_{2,4}x_{3,4}$.
\end{example}
\subsection{Special elements} Let $\bN^{\infty}$ be the set of infinite tuples ${\alpha} = (\alpha_1, \alpha_2, \ldots)$ with each $\alpha_i \in \bN$ and $\alpha_r = 0$ for $r \gg 0$. We usually suppress the trailing zeros. 
For $i \in \bN_{>0}$, we denote by ${e}_i$ the tuple satisfying $({{e}_i})_j = \delta_{i, j}$ for all $j$.

The degree of $\alpha \in \bN^{\infty}$ is $|\alpha| \coloneqq \sum \alpha_i$. Denote by $x^{{\alpha}}$ the monomial $\prod x_i^{\alpha_i} \in S$. The \textbf{power sum} $M_{\alpha}$ is the element $\gamma^1(x^{\alpha}) \times \gamma^{p-1}(1) \in \Gamma^p(S)$. It is homogeneous of degree $|\alpha|$. 

The $i$-th \textbf{elementary symmetric polynomial} in the variable $x_j$ is $E_i(x_j) = \gamma^{i}(x_j) \times \gamma^{p-i}(1) \in \Gamma^p(S)$. Given $\alpha \in \bN^{\infty}$ with $|\alpha| \leq p$, the \textbf{elementary multisymmetric polynomial} $E_{\alpha} = (\times_{i=1}^{\infty} \gamma^{\alpha_i}(x_i)) \times \gamma^{p - |\alpha|} (1)$. Note $E_i(x_j) = E_{i e_j}$.

\subsection{Generators}
For a quick reference on the theory of $\GL$-algebras over $k$, we refer the reader to \cite[Section~2]{gan22ext}. Recall that every $\GL$-algebra has a canonical $\bN$-grading; the $\GL$-action preserves this grading. For the $\GL$-algebras $\Gamma^p(S) \subset S^{\otimes p}$, the canonical grading is the same as the natural grading where $\deg(x_{i, j}) = 1$ for all $i \in [p], j \in \bN_{>0}$. 

A $\GL$-algebra is \textbf{finitely generated} if it is generated as an ordinary $k$-algebra by the $\GL$-orbit of finitely many homogeneous elements. A finitely generated $\GL$-algebra is \textbf{noetherian} if every $\GL$-stable ideal is generated by the $\GL$-orbit of finitely many homogeneous elements.

We recall a minimal generating set for $\Gamma^p(S)$ from \cite[Corollary~8.5]{ryd07gens}.
\begin{theorem}\label{thm:rydh}
    The image of the elements in the set
    \[
   \{ \, M_{\alpha} \mid \alpha \in \bN^{\infty} \text{ with } \alpha_i < p \text{ for all } i \, \} \, \cup \, \{\,E_p(x_j) \mid j \in \bN_{>0} \, \}
    \]
    form a $k$-basis of $\Gamma^p(S)_{>0}/(\Gamma^p(S)_{>0})^2$.
\end{theorem}
So $\Gamma^p(S)$ is not finitely generated as a $\GL$-algebra since 
$\Gamma^p(S)_{>0}/(\Gamma^p(S)_{>0})^2$ is nonzero in every positive degree. We now define the main object of interest, evidently a finitely generated $\GL$-algebra.
\begin{defn}
The \textbf{algebra of polarizations} $P$ is the $\GL$-subalgebra of $\Gamma^p(S)$ generated by the elements $E_i(x_1)$ with $i \in [p]$.
\end{defn}
Although we won't use this, it is not hard to see $P$ is also the algebra generated by all elementary multisymmetric polynomials, or all power sums $M_{\alpha}$ with $|\alpha| \leq p$.

\begin{example}\label{exmp1} We return to our $p=3$ running example.
Using the action of $\fS_{\infty} \subset \GL$, we get that $P$ contains $E_j(x_i)$ for all $j \in [3]$ and $i \in \bN_{>0}$. 
In particular, any multisymmetric polynomial in which only one ``column" of variables appears lies in $P$. 
This includes $M_{re_s} = x_{1,s}^r + x_{2,s}^r + x_{3,s}^r$ for all $r,s \in \bN_{>0}$.
\end{example}

\section{The main results}
\subsection{Newton's identities}
We first prove a mild variant of Newton's identities in this subsection.
In $\bZ[x_1, x_2, \ldots, x_p, t_1, t_2, \ldots t_p]$, consider the power sums with the variables $t_i$ attached, i.e.,
$\widetilde{M}_{r} = t_1 x_1^r + t_2 x_2^r + \ldots + t_p x_p^r$, and the usual elementary symmetric polynomials in the $x$ variables, i.e., $E_r(x) = \sum_{j_1 < j_2 < \ldots < j_r} x_{j_1} x_{j_2} \cdots x_{j_r}$.

\begin{proposition}\label{prop:newt} We have $\widetilde{M}_{p} = \sum_{i=0}^{p-1} (-1)^i E_{p-i}(x) \widetilde{M}_{i}.$
\end{proposition}
\begin{proof}
  Let $T$ be a formal variable and consider the equation 
  \[\prod_{i \in [p]} (T - x_i) = \sum_{i=0}^p (-1)^{p-i} E_{p-i}(x) T^i.\] 
  By substituting $T = x_j$ for $j \in [p]$ and multiplying by $t_j$, we get the equation
   \[ 0 = \sum_{i=0}^p (-1)^{p-i} E_{p-i}(x) (x_j^i t_j).\]
   Summing over all $j \in [p]$ and rearranging, we obtain the requisite identity.
\end{proof}
We will require the following application of the above result.
\begin{corollary}\label{cor:multinewt}
   Given a tuple $\alpha \in \bN^{\infty}$ with $\alpha_r \geq p$ , we have
  \[ M_{\alpha} =  \sum_{i=0}^{p-1} E_{p - i} (x_r) M_{\alpha - p e_r + i e_r}. \]
\end{corollary}
\begin{proof}
For a tuple $\beta \in \bN^{\infty}$ and $i \in [p]$, we write $x_i^{\beta}$ for the element $\prod x_{i, r}^{\beta_r} \in S^{\otimes p}$.
Define the algebra map $\phi \colon Q \to S^{\otimes p}$ by mapping $x_i \mapsto x_{i, r} \text{ and } t_i \mapsto  x_i^{\alpha - p e_r}$ for all $i \in [p]$; the element $E_i(x)$ is mapped to $E_i(x_r)$ under $\phi$, and similarly $\widetilde{M_i} \in Q$ to $M_{\alpha - p e_r + i e_r}$. The required identity is $\phi$ applied to the identity from Proposition~\ref{prop:newt}.
\end{proof}

\begin{example}\label{exmp2} When $p=3$, the above result implies that $M_{(3,3)} = x_{1,1}^3 x_{1,2}^3 + x_{2,1}^3 x_{2,2}^3 + x_{3,1}^3 x_{3,2}^3$ is in the $k$-algebra generated by the elements $M_{(3,2)}, M_{(3,1)}$ and the elementary symmetric polynomials (the above identity can easily be checked by hand in this case).
\end{example}

\subsection{Flattening} In this subsection, we study the action of $\GL$ on the power sums.
\begin{lemma}\label{lem:symr}
The $k$-linear map $\Sym^r(\bV) \to \Gamma^p(S)_r$ where $f \mapsto \gamma^{1}(f) \times \gamma^{p-1}(1)$ is an injective map of $\GL$-representations.
\end{lemma}
\begin{proof}
Let $g \in \GL$. Assume $f \in \Sym^r(\bV)$ is that uses the variables $x_1, x_2, \ldots, x_n$. Using the definition of the shuffle product, we have $\gamma^1(f) \times \gamma^{p-1}(1) = f(x_{1,1}, x_{1,2}, \ldots, x_{1,n}) + f(x_{2,1}, x_{2,2}, \ldots, x_{2, n}) + \ldots + f(x_{p,1}, x_{p,2}, \ldots, x_{p, n})$. From this observation, it is easy to now see that $\gamma^1(g(f)) \times \gamma^{p-1}(1) = g (\gamma^1(f) \times \gamma^{p-1}(1))$.
\end{proof}

The next corollary requires $\chark(k) = p$ assumption.
\begin{corollary} \label{cor:i<p}
Let $\alpha \in \bN^{\infty}$ be a tuple with $\alpha_n = 0$, $j \in \bN$, and $i \in [p-1]$. The $\GL$-subrepresentation of $\Gamma^p(A)$ generated by $M_{\alpha + (jp+i)e_n}$ contains $M_{\alpha + jpe_n + i e_{n+1}}$.
\end{corollary}
\begin{proof}
By the previous lemma, it suffices to show that the $\GL$-representation generated by $x^{\alpha + (jp+i)e_n}$ in $\Sym(\bV)$ contains $x^{\alpha + jpe_n + ie_{n+1}}$.
Assume $\alpha_i = 0$ for all $i > N \gg n$ and $r = |\alpha| + jp + i$. The map $\GL_N \to \GL(\Sym^r(k^N))$ induces a map on Lie algebras, which sends the element $E_{a, b} \in \fg\fl_N$ to the element $x_b \frac{\partial}{\partial x_a}$ for all $a, b \in [N]$. Applying $\frac{1}{i!}(E_{n, n+1})^i \in U(\fg\fl_n)$ to $x^{\alpha + (jp+i)e_n}$, we get that $x^{\alpha + (jp)e_n + ie_{n+1}}$ lies in the $\fg\fl_N$-subrepresentation generated by $x^{\alpha + (jp+i)e_n}$, which in turn is contained in the $\GL_N$-subrepresentation it generates.
\end{proof}

\begin{example}\label{exmp3}
When $p=3$, the previous result shows that $M_{(3,2)}$ is in the $\GL$-representation generated by $M_{(5)}$ (the tuple $(5)$ can be ``flattened" to $(3, 2)$). We also have $M_{(3,1)} \in \GL \langle M_{(4)} \rangle $.
\end{example}

\subsection{Boundedness of \texorpdfstring{$p$}{p}-th rootedness} We prove Theorem~\ref{thm:pthroot} in this subsection.
For a tuple $\alpha \in \bN^{\infty}$, the \textbf{length} of $\alpha$ is 
\[\len(\alpha) \coloneqq \min \{i | \alpha_j = 0 \text{ for all } j > i \}.\] 
Our main technical result is:
\begin{proposition}\label{prop:main}
Assume $\alpha \in \bN^{\infty}$ satisfies $p \divides \alpha_i$ for all $i < \len(\alpha)$. The element $M_{\alpha} \in P$.
\end{proposition}
\begin{proof}
We proceed by (outer) induction on $r \coloneqq \len(\alpha)$ with the $r=1$ case being the fact that the elementary symmetric polynomials $E_1(x_1), E_2(x_1), \ldots E_p(x_1)$ generate the ring of symmetric polynomials in $p$ variables. Now assume $ r > 1$. 
We now induct on $\alpha_r$. First, assume $\alpha_r < p$. 
The tuple $\beta = (\alpha_1, \alpha_2, \ldots, \alpha_{r-2}, \alpha_{r-1} + \alpha_r)$ is such that $p \divides \beta_i = \alpha_i$ for all $i < \len(\beta) = r-1$. So by the (outer) induction hypothesis, the element $M_{\beta} \in P$. 
By Corollary~\ref{cor:i<p}, the element $M_{\alpha} \in \GL\langle M_{\beta} \rangle$ and therefore also lies in $P$ as $P$ is closed under the $\GL$-action.
Now, assume $\alpha_r \geq p$. By Corollary~\ref{cor:multinewt}, the element
$M_{\alpha}$ 
lies in the $k$-subalgebra generated by $M_{\alpha - pe_r + i e_r}$ with $i < p$ and the elementary symmetric polynomials in the variable $x_r$. The elements $M_{\alpha - pe_r + ie_r}$ for $i < p$ lie in $P$ by the inner induction hypothesis, and all elementary symmetric polynomials lie in $P$ by definition. So $M_{\alpha} \in P$ since $P$ is a subalgebra. 
\end{proof}

\begin{example}
We demonstrate the iterative process in play above with our running example.
   We have seen that $M_{(3,3)}$ lies in the $k$-algebra generated by the elementary symmetric polynomials, $M_{(3,2)}$ and $M_{(3,1)}$ (\ref{exmp2}); that $M_{(3,2)}$ and $M_{(3,1)}$ are in the $\GL$-representation generated by $M_{(5)}$ and $M_{(4)}$ (\ref{exmp3});
  and $M_{(5)}$ and $M_{(4)}\in P$ (\ref{exmp1}). Putting everything together, we see that $M_{(3,3)} = M_{(1,1)}^3$ lies in $P$. (Note that $M_{(1,1)}$ already lies in $P$, so this example is meant solely for illustrative purposes.)
\end{example}

Theorem~\ref{thm:pthroot} easily follows from the observation that $M_{\alpha}^p = M_{p \alpha}$ in char $p$.
\begin{proof}[Proof of Theorem~\ref{thm:pthroot}]
   It suffices to show that the $p$-th power of a generating set of $\Gamma^p(S)$ lies in $P$. We use the generating set from Theorem~\ref{thm:rydh}. The elements $\gamma^p(x_i)$ already lie in $P$, so it suffices to show $M_{\alpha}^p$ lies in $P$ for arbitrary $\alpha$. But $M_{\alpha}^p = M_{p \alpha}$, which lies in $P$ by Proposition~\ref{prop:main}.
\end{proof}

\subsection{Non-noetherianity of \texorpdfstring{$P$}{P}}
It suffices to prove Theorem~\ref{thm:glalg} after passing to the algebraic closure of $k$, since if $\overline{k} \otimes P$ is not noetherian, then neither is $P$ (we are implicitly using the fact that the category $\Pol_k$ behaves well under base change along field extensions). So, we further assume that $k$ is algebraically closed and, in particular, perfect.

The final ingredient we need is the notion of Frobenius splitting. 
For any $k$-algebra, we have the \textbf{Frobenius map} $F \colon T \to T$ where $f \mapsto f^p$ for all $f \in T$. We say $T$ is \textbf{$F$-split} if there exists a left inverse $\Psi$ to $F$ that satisfies $\Psi(a^p b) = a \Psi(b)$ and $\Psi(a + b) = \Psi(a) + \Psi(b)$ for all $a, b \in T$. The next result is well-known and first appeared in \cite[Page~77]{hh92bigcm} to our knowledge.

\begin{lemma}\label{lem:Fsplmon}
    Let $V$ be a permutation representation of a finite group $G$ over $k$. The invariant ring $\Sym(V)^G$ is $F$-split and sends homogeneous elements of positive degree to homogeneous elements of positive degree (including possibly, the zero element).
\end{lemma}
\begin{proof}
Let $\{b_i \}$ be a $k$-basis of $V$ which is permuted by $G$. The monomials in the $b_i$ span $\Sym(V)$. Given a monomial $m$, let $T_m$ denote the sum of all monomials in the $G$-orbit of $m$. The elements $T_m$ form a $k$-basis of $\Sym(V)^G$. Define $\Psi$ by mapping $T_m$ to $T_n$ if $m = n^p$ for some monomial $n$, and otherwise to $0$. 
It preserves homogeneity. We leave the verification that this is an $F$-splitting to the reader. 
\end{proof}
We emphasize that neither the Frobenius map nor an $F$-splitting (if it exists) is a map of $\GL$-representations. 
\begin{proof}[Proof of Theorem~\ref{thm:glalg}]
Consider the $\GL$-stable ideal $I \subset P$ generated by the elements $M_{\alpha}^p$ for all nonzero $\alpha \in \bN^{\infty}$; these elements belong to $P$ by Theorem~\ref{thm:pthroot}.
Assume for the sake of contradiction that $I$ is finitely generated, say by the $\GL$-orbit of elements of degree $\leq pd$ for some $d$. 
Fix an integer $N > d$ and let $\omega = \sum_{i=1}^N e_i \in \bN^{\infty}$.
Since $M_{\omega}^p \in I$, we obtain an expression
\begin{align*}
M_{\omega}^p &=  g_1(M_{\alpha^1}^p)f_1 +  g_2(M_{\alpha^2}^p) f_2 + \ldots +  g_r (M_{\alpha^r}^p) f_r \\
&=  (g_1(M_{\alpha^1}))^p f_1 + (g_2(M_{\alpha^2}))^p f_2+ \ldots +  (g_r (M_{\alpha^r}))^pf_r
\end{align*}
for tuples $\alpha^i \in \bN^{\infty}$ with $|\alpha^i| \leq d$, homogeneous elements $f_i \in P_{>0}$, and $g_i \in \GL$ for all $i \in [r]$.
Let $\Psi$ be an $F$-splitting of $\Gamma^p(S)$ which sends homogeneous elements of positive degree to homogeneous elements of positive degree. Applying $\Psi$ to the above equation and using $\Psi(a^b p) = a \Psi(b)$, we have
\[
M_{\omega} =  g_1(M_{\alpha^1}) \Psi(f_1) +  g_2(M_{\alpha^2}) \Psi(f_2)+ \ldots +  g_r (M_{\alpha^r})\Psi(f_r)
\]
in $\Gamma^p(S)$. Since $f_i \in P_{>0}$, the element $\Psi(f_i) \in \Gamma^p(S)_{>0}$ for all $i \in [r]$. So the right side of the equation lies in $(\Gamma^p(S)_{>0})^2$. This contradicts the fact that $M_{\omega} \notin (\Gamma^p(S)_{>0})^2$ by Theorem~\ref{thm:rydh}. Therefore the $\GL$-stable ideal $I$ is not finitely generated, and so $P$ is not noetherian.
\end{proof}

\subsection{Concluding remarks}\label{ss:rmks}
\begin{enumerate}
    \item The representation $P_{>0}/(P_{>0})^2$ is supported in degrees $1, 2, \ldots, p$. We currently do not know whether $\GL$-algebras generated in degrees $< p$ are noetherian. 
    See \cite[Remark~3.4]{gan24non} for a potential approach to prove non-noetherianity of $\Sym(\lw^2(\bV))$ for odd $p$.
    \item Let $g(n) = \min \{d \in \bN \vert \exists f \in \Gamma^p(S)_{> n} \setminus (\Gamma^p(S)_{>0})^2 \text{ and } f^{p^d} \in P\}$. 
    Our proof of Theorem~\ref{thm:glalg} using Frobenius splittings would have not worked only if $\liminf_{n \to \infty} g(n) = \infty$.
    \item Consider the tensor ideals $F_1$ and $F_2$ of $\Pol_k$ generated by Frobenius-twisted representations and twice Frobenius-twisted representations respectively. Our counterexample persists in $\Pol_k/F_2$ but collapses in $\Pol_k/F_1$. The latter category is equivalent to the category of sequences of representations of the symmetric groups endowed with the induction product. The noetherian problem for twisted commutative algebras (TCAs), i.e.,~commutative algebra objects in $\Pol_k/F_1$, remains open in all characteristics.  
    \item 
    In \cite{gan25nc}, we showed that finite group invariants of the TCAs corresponding to $\Sym(k^p \otimes k^{\infty})$ are also not finitely generated in the modular setting, providing an analogue of Theorem~\ref{thm:rydh}. 
    The key difference in $\Pol_k$ seems to be our ability to use the Frobenius map to ``push" the infinite generators of this invariant subrings into a finitely generated subalgebra (made precise by Theorem~\ref{thm:pthroot}). We are unaware of any operations on TCAs that may permit us to execute a similar idea.
\end{enumerate}

\bibliographystyle{alpha}
\bibliography{bibliography}
\end{document}